\documentclass[12pt,a4paper]{amsart}

\usepackage{amssymb,color,currfile,pdfpages,mathtools}

\usepackage[english]{babel}
\usepackage{verbatim,here}
\usepackage[T1]{fontenc}
\usepackage{epstopdf}
\usepackage{floatflt,graphicx}
\usepackage{color,xcolor}
\usepackage{enumerate}
\usepackage{animate}
\usepackage{a4wide}
\usepackage{amsmath, amssymb}
\usepackage{amsthm}
\usepackage{float}
\usepackage{pdfpages}
\newcounter{minutes}
\divide\time by 60
\newcounter{hours}
\multiply\time by 60 \addtocounter{minutes}{-\time}

\DeclareMathOperator{\dist}{dist}

\font\fFt=eusm10 
\font\fFa=eusm7  
\font\fFp=eusm5  
\def\K{\mathchoice
{\hbox{\,\fFt K}}
{\hbox{\,\fFt K}}
{\hbox{\,\fFa K}}
{\hbox{\,\fFp K}}}

\dedicatory{}
\commby{}
\theoremstyle{plain}
\newtheorem{thm}[equation]{Theorem}

\newtheorem{lem}[equation]{Lemma}

\newtheorem{prop}[equation]{Proposition}

\theoremstyle{definition}

\theoremstyle{remark}

\newtheorem{nonsec}[equation]{}

\numberwithin{equation}{section}

\newcommand{\beq}{\begin{equation}}
\newcommand{\eeq}{\end{equation}}
\newcommand{\ben}{\begin{enumerate}}
\newcommand{\een}{\end{enumerate}}
\newcommand{\bequu}{\begin{eqnarray*}}
\newcommand{\eequu}{\end{eqnarray*}}
\newcommand{\bequ}{\begin{eqnarray}}
\newcommand{\eequ}{\end{eqnarray}}

\newcommand{\sh}{\,\textnormal{sh}}

\renewcommand{\th}{\,\textnormal{th}}

\begin{document}
\thispagestyle{empty}
\def\thefootnote{}


\title[quasiregular mappings of the half plane]
{On distortion of quasiregular mappings\\ of the upper half plane }


\author[M. Fujimura]{Masayo Fujimura}
\author[M. Vuorinen]{Matti Vuorinen}

\date{}

\begin{abstract}
We  prove a sharp result for the  distortion of a hyperbolic type metric 
under $K$-quasiregular mappings of the upper half plane. The proof makes use
of a new kind of Bernoulli inequality and the Schwarz lemma for quasiregular
mappings. 
\end{abstract}



\maketitle
\noindent \textbf{Author information.}\\
Masayo Fujimura$^1$, email: \texttt{masayo@nda.ac.jp}, ORCID: 0000-0002-5837-8167\\
Matti Vuorinen$^2$, email: \texttt{vuorinen@utu.fi}, ORCID: 0000-0002-1734-8228\\
1: Department of Mathematics, National Defense Academy of Japan, Yokosuka, Japan\\
2: Department of Mathematics and Statistics, University of Turku, FI-20014 Turku, Finland\\
\textbf{Data availability statement.} Not applicable, no new data was generated.\\
\textbf{Conflict of interest statement.} There is no conflict of interest.\\


\footnotetext{\texttt{{\tiny File:~\jobname .tex, printed: \number\year-%
\number\month-\number\day, \thehours.\ifnum\theminutes<10{0}\fi\theminutes}}}
\makeatletter

\makeatother



\section{Introduction}
Alongside the hyperbolic metric, several  other metrics similar to it have
become standard tools of geometric function theory  \cite{dhv,fmv,gh,ha,h, hkv}.
In \cite[pp.42-48]{p} the author lists twelve metrics recurrent in complex
analysis.
These metrics, sometimes called metrics of  hyperbolic type, are typically  
not M\"obius invariant, but they are often quasi-invariant
and within a constant factor from the hyperbolic metric.
Here we study one of these metrics, defined on a proper subdomain   $D$ 
of a metric space $(X, d)$ as follows
\begin{equation}
\label{hc}
h_{D,c}(x,y) = \log\left(1+c\frac{d(x,y)}{\sqrt{d_D(x)d_D(y)}}\right),\quad c>0,
\end{equation}
where $d_{D}(x)=\dist (x, \partial D).$ 

This metric has been studied in several very recent papers.
In  \cite{ra},  the metric \eqref{hc} was called the {\it geometric mean distance metric}. In \cite{ra} also comparison inequalities between  the metric \eqref{hc}
and some other metrics were given and  the geometry defined by the metric studied. 
For further work, see also \cite{wwjz,zzpg}.

\begin{thm}\label{th1} (\cite[Thm 1.1]{dhv})
 The function \eqref{hc}
is a metric for every $c \geq 2$. The constant $2$ is best possible here.
\end{thm}

\medskip

Here the constant $c$ is sharp in the case if the domain $D$ is 
the unit disk $\mathbb{B}^2.$ It was shown in \cite[Rem. 3.29]{dhv} 
that $h_{\mathbb{B}^2,c }$ is not a metric if $0<c<2.$ For the half plane 
${\mathbb{H}^2}$ case we prove the following result.

\medskip

\begin{thm}\label{th2}
 The function  $h_{\mathbb{H}^2,c }$
is a metric for every $c \ge 1$. 
\end{thm}

\medskip

The behavior of this metric under quasiconformal mappings was also studied  in
\cite[Thm 4.9]{dhv}. Our goal here is to prove the following new sharp result
for quasiregular mappings of the upper half plane ${\mathbb{H}^2}.$

\medskip

\begin{thm}\label{th3} For $K \ge 1$  there exists a constant $\lambda(K) \in [1, \exp(\pi(K-1/K)))$ 
such that for a $K$-quasiregular mapping
$f: {\mathbb{H}^2} \to {\mathbb{H}^2} =f( {\mathbb{H}^2})$ and
for all $x,y \in {\mathbb{H}^2}, c \ge 1,$
\[
h_{\mathbb{H}^2, c}(f(x),f(y))\le \lambda(K)^{1/2} \,K^{1+c}\, 
\max\{ h_{\mathbb{H}^2, c}(x,y)^{1/K} ,h_{\mathbb{H}^2, c}(x,y)\}\,.
\]
\end{thm}

\medskip
Theorem   \ref{th3}
is perhaps new also in the case when $K=1, c>1,$ in which case the constant
$ \lambda(K)^{1/2} \,K^{1+c} = 1,$ and the function $f$ is analytic. 
The constant $\lambda(K)$ is explicitly given in \eqref{lam} with concrete estimates.
The proof of Theorem \ref{th3} is based on two main components. The first component
is the well-known Schwarz lemma for quasiregular mappings and associated estimates for special functions expressed in terms of complete elliptic integrals. The second component is a Bernoulli type inequality, formulated as Theorem \ref{fuji} that enables us to simplify the inequalities given by the Schwarz lemma. 


\section{Preliminary results}
%

\begin{nonsec}{\bf Hyperbolic geometry.}\label{hg}
We recall some basic formulas and notation for hyperbolic geometry from \cite{b}.
The hyperbolic metrics of the unit disk ${\mathbb{B}^2}$ and
the upper half plane  ${\mathbb{H}^2}$ are defined, resp., by
\begin{equation}\label{rhoB}
\sh \frac{\rho_{\mathbb{B}^2}(a,b)}{2}=
\frac{|a-b|}{\sqrt{(1-|a|^2)(1-|b|^2)}} ,\quad a,b\in \mathbb{B}^2\,,
\end{equation}
and
\begin{equation}\label{rhoH}
{\rm ch}\rho_{\mathbb{H}^2}(x,y)=1+ \frac{|x-y|^2}{2 {\rm Im}(x) {\rm Im}(y)}
\, ,\quad x,y\in \mathbb{H}^2\,.
\end{equation}

Above the symbols ${\rm sh}$ and ${\rm ch}$ stand for the hyperbolic
sine and cosine functions. Their inverses are  ${\rm arsh}$ and ${\rm arch}.$ 
Recalling that  $2({\rm ch} \,t -1)=(e^{\frac{t}2}-e^{-\frac{t}2})^2$ for $t\geq 0,$
 we obtain
by \eqref{rhoH} and the definition \eqref{hc} 
\begin{equation} \label{hc3}
2\, {\rm sh} \frac{\rho_{\mathbb{H}^2}(x,y)}{2}=\sqrt{2\,({\rm ch} \rho_{\mathbb{H}^2}(x,y)-1)} = \frac{|x-y|}{\sqrt{{\rm Im}(x){\rm Im}(y)}} = \frac{1}{c}\left( e^{h_{{\mathbb{H}^2},c}(x,y)} -1\right)\,.
\end{equation}


These two metrics are M\"obius invariant: if $G, D \in \{ \mathbb{B}^2, \mathbb{H}^2\}$
and $f:G \to D= f(G)$ is a M\"obius transformation, then  
$\rho_G(x,y)= \rho_D(f(x),f(y))$ for all $x,y \in G.$
\end{nonsec}

By the Bernoulli inequality \cite[(5.6)]{hkv} we have for   
$c_1\geq c_2 \geq 1$ and all $t>0$
\begin{equation} \label{berineq}
\log(1+ c_1 t) \leq \frac{c_1}{c_2} \,\log(1+ c_2 t)\,,
\end{equation}
and hence for all $x,y \in D$
\begin{equation}
h_{D,c_1}(x,y) \leq \frac{c_1}{c_2} \, h_{D,c_2}(x,y) \leq \frac{c_1}{c_2} \, h_{D,{c_1}}(x,y)\, .
\end{equation}

\medskip
We record the next three auxiliary results for the proof of Theorem \ref{th2} and
give the proof based on these results.

\noindent

\begin{prop}\label{Lemma1}
For $ c\geq1 $ and $ x\geq 1 $, the following inequality holds
$$
     x \leq c\Big(x-\frac1x\Big)+1.
$$
\end{prop}

\begin{proof} The claim is equivalent to $x-1 \leq c (x-1)(x+1)/x$ which clearly holds
for $x\geq 1, c\geq 1 .$
\end{proof}

\medskip

\noindent

\begin{lem}\label{Lemma2}
For $ c\geq 1 $, the following function is decreasing for $ x> 1 $,
$$
      f(x)=\frac{\log \Big(1+c\big(x-\frac1x\big)\Big)}{\log x}.
$$
\end{lem}

\medskip

\begin{proof}

Consider the derivative of $ f $,
\begin{align} \notag
   f'(x)& =\dfrac{\frac{c(1+\frac{1}{x^2})}{1+c(x-\frac{1}{x})}\log x
                -\frac{1}{x}\log(1+c(x-\frac{1}{x}))}
              {(\log x)^2} \\ \label{eq:f1}
        &  =\dfrac{c(x^2+1)\log x-(x+c(x^2-1))\log(1+c(x-\frac1x))}
                  {x^2(1+c(x-\frac{1}{x}))(\log x)^2}.
\end{align}
Let $ f_1(x) $ be the numerator of \eqref{eq:f1}.
Since $ x^2(1+c(x-\frac{1}{x}))(\log x)^2>0 $, 
we have to check whether $ f_1(x)<0 $ for $ x>1 $.

Now we have $ f_1(1)=0-0=0 $ and
\begin{equation}\label{eq:f1'}
   f_1'(x)=2cx\log x-(2cx+1)\log\Big(c\big(x-\frac1x\big)+1\Big).
\end{equation}
As $ 1\leq x\leq c(x-\frac1x)+1 $ holds by Proposition \ref{Lemma1},
we have $ f_1'(x)<0 $.

Therefore $ f_1(x)<0 $ holds and the assertion is obtained.
\end{proof}

\medskip

\noindent

\begin{lem} \label{MFprop} For a constant $ c $ with $ c\geq 1 $ and $t\geq 0$, let 
$$ \displaystyle
   F(t)=\log\big(1+c\sqrt{2({\rm ch}\, t -1)}\big) \,. $$
Then $F(0)= 0$ and $F: [0,\infty) \to  [0,\infty)$ is increasing whereas
$ {F(t)}/{t} $ is decreasing for $ t > 0 $.
\end{lem}
\bigskip

\begin{proof}  Because 
$2({\rm ch} \,t -1)=(e^{\frac{t}2}-e^{-\frac{t}2})^2$ for $t\geq 0,$
the function $ F(t)/t $ can be expressed as
$$
   \dfrac{F(t)}{t}=\dfrac{\log\big(1+c\sqrt{2({\rm ch}\, t -1)}\big)}{t}
                  =\dfrac{\log\big(1+c(e^{\frac{t}2}-e^{-\frac{t}2})\big)}{t}\,, \quad t > 0.
$$ 
So, setting $ e^{\frac{t}2}=x $, we have
$$
    \dfrac{F(t)}{t}
      =\dfrac{F(2\log x)}{2\log x}=\dfrac{\log(1+c(x-\frac{1}x))}{2\log x}.
$$
From the monotonicity of the exponential function and Lemma \ref{Lemma2},
the assertion is obtained. 
\end{proof}

 \begin{nonsec}{\bf Proof of Theorem \ref{th2}.} Suppose that
  $g : [0,\infty)\to [0,\infty)$ is an increasing function
with $g(0)=0$ such that $g(t)/t$ is decreasing on $(0,\infty)$. Then by \cite[p.80, Ex 5.24(1)]{hkv} the subadditivity inequality $g(s+t)\le g(s)+g(t)$ holds for $s,t\ge 0$.
We now apply this observation with the function $F$ in place of $g.$ 
By  Lemma \ref{MFprop} we see that the function $F$ satisfies, indeed,
 the above requirements for the function $g.$ 
By   \eqref{rhoH} and the definition \eqref{hc} we can write
\[
h_{\mathbb{H}^2,c}(x,y)= F(\rho_{\mathbb{H}^2}(x,y))\,.
\]
The triangle inequality follows now from the above subadditivity property of the function $g.$
\hfill $ \square $
\end{nonsec}

\medskip

In the next theorem we compare the metric $h_{\mathbb{H}^2,c}$ to the hyperbolic metric. To this end, the following inequality is needed \cite[(4.12)]{hkv}
\begin{equation} \label{hkv412}
2 \log \left(1+ \sqrt{\frac{1}{2}(x-1)} \right)\le {\rm arch}\, x \le 2 \log \left(1+ \sqrt{2(x-1)} \right), \quad x \ge 1\,.
\end{equation}

\medskip

\begin{thm}\label{CompRho} For all $x,y \in \mathbb{H}^2 $
and all $c \ge 1$ we have
$$
 h_{\mathbb{H}^2,c}(x,y)/c \leq \rho_{\mathbb{H}^2}(x,y) \leq 2  h_{\mathbb{H}^2,c}(x,y) \,.
$$
\end{thm}

\begin{proof}
For convenience of notation write $D= \mathbb{H}^2.$ 
Fix $x,y \in D$ and $c \ge 1\,.$ Applying \eqref{hc3}, \eqref{hkv412}, 
and the Bernoulli inequality \eqref{berineq} we now have
\[
\rho_D(x,y)= {\rm arch}(1+ \frac{1}{2c^2}(e^{h_{D,c}(x,y)}-1)^2)\ge 2 \log(1+ \frac{1}{2c}(e^{h_{D,c}(x,y)}-1)) \ge\]
\[
\frac{1}{c} \log(1+ (e^{h_{D,c}(x,y)}-1))=h_{D,c}(x,y)/c
\]
which proves the first inequality.

To prove the second inequality, it is enough to prove the case $c=1$ because 
\( h_{D,1}(x,y) \le h_{D,c}(x,y)\) for \( c \ge 1 \,.\) Again by \eqref{hkv412}
\[
\rho_D(x,y)= {\rm arch}(1+ \frac{1}{2}(e^{h_{D,1}(x,y)}-1)^2)\le 2 \log(1+ (e^{h_{D,1}(x,y)}-1)) = 2 \, h_{D,1}(x,y)\]
completing the proof.
\end{proof}


\section{ A Bernoulli type inequality}
The following theorem yields a Bernoulli type inequality 
which is one of the main steps in the proof of Theorem \ref{th3}.
The proof of the Bernoulli type inequality is based on three lemmas, 
Lemmas \ref{LemmaA} A, \ref{LemmaB} B, and 
\ref{LemmaC} C, formulated below.

\noindent
\begin{thm}\label{fuji} The following inequality
$$
  \log\big(1+2c\max\big\{t^K,t^{1/K}\big\}\big)\leq
     K^{1+c}\max\big\{\log(1+2ct),\big(\log(1+2ct)\big)^{1/K}\big\}
$$
holds for $ c\geq 1, \ K\geq 1,\ t>0 $.
\end{thm}

\medskip

It will be shown in Remark \ref{rmk310} that Theorem \ref{fuji} is sharp in the sense that $K^{1+c}$ cannot be replaced with $K^2.$
\noindent
\begin{lem}{\bf  A.} \label{LemmaA}\quad
For $ c\geq 1,\ K\geq 1 $,
$$
    \big(K^{1+c}\big)^{\frac{K}{K-1}}-2c>0.
$$
\end{lem}

\medskip

\begin{proof} Let $ A(K)=\big(K^{1+c}\big)^{\frac{K}{K-1}}-2c $.
Differentiation yields
$$
   A'(K)=\big(K^{1+c}\big)^{\frac{K}{K-1}}\frac{1}{(K-1)^2}
      (1+c)\big((K-1)-\log K\big).
$$
Since $ K\geq 1 $, the inequality $ (K-1)-\log K>0 $ holds,
and we have $ A'(K)>0 $.

Next, we will check $ A(1)>0 $.
Set $ y=\big(K^{1+c}\big)^{\frac{K}{K-1}} $ and consider
$ \log y=\frac{K}{K-1}(1+c)\log K $.
Then,
$$
     \lim_{K\to 1} \log y
    = \lim_{K\to 1} (1+c)\frac{K\log K}{K-1}
    = \lim_{K\to 1} (1+c)\frac{(K\log K)'}{(K-1)'}
    = \lim_{K\to 1} (1+c)\frac{\log K+1}{1}
    =1+c. 
$$
So we have
$$
          \lim_{K\to 1}  y
         =\lim_{K\to 1} \big(K^{1+c}\big)^{\frac{K}{K-1}}=e^{1+c}.
$$
Here, for $ c>1 $, the inequality $ e^{1+c}-2c>0 $ holds.
Therefore, 
$$
     A(1)=\lim_{K\to 1}\big(K^{1+c}\big)^{\frac{K}{K-1}}-2c
         =e^{1+c}-2c>0.
$$
As  we have $ A(1)>0 $ and $ A'(K)>0 $ for $ K>1 $,
the assertion $ A(K)> 0 $ is obtained.
\end{proof}

\medskip

\begin{lem} {\bf  B.} \label{LemmaB} (1)
For $  K\geq 1$  and $t>0 $, the function
$$
   B_1(t)=t^{1/K}\log t
$$
attains its minimum $ -\frac{K}e $ at $ t=e^{-K}$.

(2) For $ K\geq 1, c\geq 1 $ and $ \dfrac{e-1}{2c}\leq t< 1 $,
the following holds
$$
  K^{1+c}\log(1+2ct)-\log(1+2ct^{1/K})>0.
$$

\end{lem}
\medskip

\noindent
\begin{proof} (1)
Consider the equation
$$
    B_1'(t)=\frac1Kt^{\frac1K-1}\log t+t^{1/K}\frac{1}t
         =t^{\frac1K-1}\Big(\frac1K\log t+1\Big)=0.
$$
Because $ t>0 $, we have $ \log t=-K $ that is $ t=e^{-K} $.
So $ B_1(t) $ attains the minimum at $ t=e^{-K} $.
Moreover, the minimum value is given by 
$$
    B_1(e^{-K})=(e^{-K})^{1/K}\log e^{-K}=-\frac{K}{e}.
$$

(2) Let $ B_2(K)= K^{1+c}\log(1+2ct)-\log(1+2ct^{1/K})>0. $
For each fixed $ c,t $ with $ c\geq 1 $, $ \dfrac{e-1}{2c}\leq t< 1 $,
we will show $ B_2(K)>0 $.
Observe that $ \log(1+2ct)>1 $ holds in this case.

Differentiation and part (1) yield
\begin{align*}
  \frac{\partial}{\partial K}  B_2(K)
     & =\frac{2ct^{1/K}\log t}{K^2(1+2ct^{1/K})}
        +(1+c)K^c\log(1+2ct) \\
     & \geq \frac{2c\big(-\frac{K}e\big)}{K^2(1+2ct^{1/K})}
        +(1+c)K^c\log(1+2ct) \qquad \textrm{(from part (1))}\\
     & \geq -\frac{2c}{Ke(1+2ct^{1/K})}+(1+c)K^c 
                \qquad (\textrm{from } \log(1+2ct)>1)\\
     & \geq -\frac{2c}{e^2}+(1+c) \qquad 
                  (\textrm{from }K^c>K\geq1,\ 1+2ct^{1/K}\geq1+2ct\geq e)\\
     & \geq -c+1+c>0.
\end{align*}
Moreover, 
$$
   B_2(1)=1\cdot\log(1+2ct)-\log(1+2ct^1)=0.
$$
As  we have $ B_2(1)=0 $ and $ \dfrac{\partial}{\partial K}B_2(K)>0 $ 
for $ K\geq1 $,
the assertion $ B_2(K)> 0 $ is obtained.

\end{proof}

\noindent
\begin{lem}
{\bf C.} \label{LemmaC}\quad
For $ K\geq 1, c\geq 1,\ t\geq 1 $, the function
$$
    C(K)= (1+2ct)^{K}-(1+2ct^K)
$$
is increasing.
\end{lem}
\medskip

\begin{proof} The derivative is
$$
    \frac{\partial}{\partial K} C(K)
        = (1+2ct)^K\log(1+2ct)-2ct^K\log t.
$$
Observing that for  $K \geq 1, t\geq 1, c\geq1 $, the following inequalites hold
$$
   (1+2ct)^K > 1+(2ct)^K > 2ct^K, \quad \textrm{and}\quad
   1+2ct> t
$$
we have
$$
   \frac{\partial}{\partial K}  C(K)  
       > (1+2ct)^K\log(1+2ct)-(1+2ct)^K\log t
       > 0
$$
and hence $C(K)$ is increasing for $K\geq 1.$
\end{proof}

\bigskip

\noindent

\begin{nonsec}{\bf Proof of Theorem \ref{fuji}}.
For $K \geq 1, c\geq 1 $ and $ t>0 $, the following relations hold
\begin{equation}\label{eq:max1}
   \max\{t^K,t^{1/K}\}=\left\{
       \begin{array}{ll}
            t^K & (t>1)\\
            t^{1/K} & (0<t\leq 1)
       \end{array} 
       \right.
\end{equation}
and
\end{nonsec}
\begin{equation}\label{eq:max2}
   \max\big\{\log(1+2ct),\big(\log(1+2ct)\big)^{1/K}\big\}=\left\{
       \begin{array}{ll}
            \log(1+2ct) & (t>\frac{e-1}{2c})\\
            \big(\log(1+2ct)\big)^{1/K} & (0<t\leq \frac{e-1}{2c}).
       \end{array} 
       \right.
\end{equation}
Because $c \geq 1,$ we have  $ 0<\frac{e-1}{2c}<1 $.
Therefore, to prove Theorem \ref{fuji}, it is sufficient to show the following inequalities
A,B and C.
\begin{enumerate}
  \item[\bf A: ] For $ c\geq 1, \ 0<t\leq\frac{e-1}{2c},\ K\geq 1 $, 
      the following inequality holds
      $$
          \log(1+2ct^{1/K})\leq K^{1+c}\big(\log(1+2ct)\big)^{1/K}.
      $$

  \item[\bf B: ] For $ c\geq 1, \ \frac{e-1}{2c}<t<1,\ K\geq 1 $, 
      the following inequality holds
      $$
          \log(1+2ct^{1/K})\leq K^{1+c} \,\log(1+2ct)\,.
      $$

  \item[\bf C: ] For $ c\geq 1, \ t\geq 1,\ K\geq 1 $, the following inequality holds
      $$
          \log(1+2ct^{K})\leq K^{1+c}\, \log(1+2ct)\,.
      $$
\end{enumerate}

Accordingly, we consider these three cases A,B,C.

\medskip

\begin{enumerate}

  \item[{\bf A: }]
     Fix $ c,K $ with $ c\geq 1  $, $ K\geq1 .$ 
     In this case, we remark that $ 0<\log(1+2ct)\leq1 $
     as $ 0<t\leq\frac{e-1}{2c} $.
     Let 
     \begin{equation}\label{eq:A}
       a(t):= K^{1+c}\big(\log(1+2ct)\big)^{1/K} -\log(1+2ct^{1/K}).
     \end{equation}
     Here, we will show  that $ a(t)\geq 0 $ holds.
     First, we remark that $ a(0)=K^{1+c}\log(1)-\log(1)=0 $.
     Next, we obtain
     \begin{equation}\label{eq:A'}
        a'(t)=K^{1+c}\cdot\frac{1}{K}\cdot
            \frac{2c\big(\log(1+2ct)\big)^{\frac{1}{K}-1}}
                                    {1+2ct}
              -\frac{1}{K}\frac{2ct^{\frac1K-1}}{1+2ct^{\frac1K}}.
     \end{equation}
     So, if we can show $ a'(t)\geq0 $ for $ 0<t\leq\frac{e-1}{2c}$,
     we obtain the assertion.
     Dividing the right side of \eqref{eq:A'} by $ \dfrac1K(2c)>0 $,
     we have
     $$
       A_0:=\frac{K}{2c}a'(t)
           =K^{1+c}\frac{\big(\log(1+2ct)\big)^{\frac{1}{K}-1}}{1+2ct}
              -\frac{t^{\frac1K-1}}{1+2ct^{\frac1K}}.
     $$
     Multiplying the above constant $ A_0 $ by $ (1+2ct)(1+2ct^{\frac1K})>0 $ we have
     $$
       A_1:=(1+2ct)(1+2ct^{\frac1K})A_0
            =K^{1+c}(1+2ct^{\frac1K})\big(\log(1+2ct)\big)^{\frac{1}{K}-1}
              -t^{\frac1K-1}(1+2ct).
     $$
     Here, remark that $ 1+2ct^{\frac1K}\geq 1+2ct $ holds
     because $ 0<t < \frac{e-1}{2c} <1, \ K>1 $. 
     Therefore,
     $$
       A_1>(1+2ct)\Big(K^{1+c}\big(\log(1+2ct)\big)^{\frac1K-1}-t^{\frac1K-1}\Big).
     $$ 
     Multiplying the above constant $ A_1 $ by 
     $ \frac{1}{1+2ct}\big(t\log(1+2ct)\big)^{\frac{K-1}{K}}>0 $,
     $$
        A_2:=K^{1+c}t^{\frac{K-1}K}-\big(\log(1+2ct)\big)^{\frac{K-1}{K}}.
     $$ 
     To show $ A_2>0 $, it is sufficient to check
     $ \Big(K^{1+c}\Big)^{\frac{K}{K-1}}t>\log(1+2ct) $.
     Because $ X-\log(X+1)>0 $ holds for $ X>0$, we have by Lemma \ref{LemmaA} A 
     $$ 
       \Big(K^{1+c}\Big)^{\frac{K}{K-1}}t-\log(1+2ct)
       \geq\Big(K^{1+c}\Big)^{\frac{K}{K-1}}t-2ct
         =t\Big(\big(K^{1+c}\big)^{\frac{K}{K-1}}-2c\Big)>0.
     $$
     Therefore, for $ 0<t<\frac{e-1}{2c} $,
     $ A_2>0$ and $ a'(t)>0 $ hold.
      Hence, 
      $$  
          \log\big(1+2ct^{1/K}\big) \leq K^{1+c}\log(1+2ct)^{1/K}
      $$
      is obtained in this case.

  \item[{\bf B: }]
     Fix $K, c $ and $t$ with $K \geq 1, c\geq 1 $ and $ \frac{e-1}{2c}<t<1 $.
     In this case, we remark that $ \log(1+2ct)>1 $. 
     By Lemma \ref{LemmaB} B (2) we see that 
    the desired inequality holds.

  \item[{\bf C: }]
     In this case, we remark that $ \log(1+2ct)>1 $.
     By Lemma \ref{LemmaC} C, the following inequality holds
     $$
       (1+2ct)^{K}-(1+2ct^{K}) \geq 0 $$
       which implies
$$        \log\big(1+2ct^{K}\big)  \leq K \,\log(1+2ct) \,.
     $$
    This inequality yields the claim C with the better constant $K$ in place of $K^{1+c}.$
     
 \end{enumerate}
In conclusion, the proof of Theorem \ref{fuji} is complete. \hfill $ \square $

\medskip

\begin{nonsec}{\bf Remark.}\label{rmk310} For $ K=1.2, c= 5,$ and $t=0.001,$ the following inequality
does not hold
\[    \log(1+2c \max\{t^K,t^{1/K}\})
          \le K^2 \max\{ \log(1+2ct),(\log(1+2ct))^{1/K}\}. \]
In conclusion, Theorem \ref{fuji} is sharp in the sense that the constant $K^{1+c}$  can not be replaced with $K^2\,.$     
%
\end{nonsec}

\section{Proof of Theorem \ref{th3}}

We first recapitulate some fundamental facts
about $K$-quasiregular mappings needed for the sequel. The definition of these mappings  is given in \cite[pp. 288-289]{hkv}, \cite[pp. 10-11]{r}. 
For the proof of the main result, Theorem \ref{th3}, we need some
properties of an increasing homeomorphism $\varphi_{K,2}:[0,1]\to[0,1]$, \cite[(9.13), p. 167]{hkv} 
\begin{align*}
\varphi_{K,2}(r)=\frac{1}{\gamma_2^{-1}(K\gamma_2(1\slash r))} =\mu^{-1}(\mu(r)/K),\quad
0<r<1,\,K>0.
\end{align*}
This function is the special function of the Schwarz
lemma in \cite[Thm 16.2]{hkv}, and its properties and estimates are 
crucial for the proof.
The Gr\"otzsch capacity, a decreasing homeomorphism $\gamma_2:(1,\infty)\to (0,\infty)$, is defined as \cite[(7.18), p. 122]{hkv}
\begin{align*}
\gamma_2(1/r)=\frac{2\pi}{\mu(r)},\quad \mu(r)=\frac{\pi}{2}\frac{\K(\sqrt{1-r^2})}{\K(r)},\quad
\K(r)=\int^1_0 \frac{dx}{\sqrt{(1-x^2)(1-r^2x^2)}}
\end{align*}
with $0<r<1$. Denote then \cite[(10.4), p. 203]{avv}
\begin{align} \label{lam}
\lambda(K)=\left(\frac{\varphi_{K,2}(1/\sqrt{2})}{\varphi_{1/K,2}(1/\sqrt{2})}\right)^2,\quad K>1.   
\end{align}
Trivially, $\lambda(1)=1$ and, 
 by \cite[Thm 10.35, p. 219]{avv},
$\lambda(K)\in [1,e^{\pi(K-1/K))}$ for $K\ge 1$. The function
$\varphi_{K,2}$ satisfies many identities and inequalities
as shown in \cite[Section 10]{avv}. In particular, for $t>0, K\ge 1,$ we have
by \cite[(10.3), 10.24]{avv}
\begin{equation}
\label{eta1}
\eta_K(t) \equiv \frac{\varphi_{K,2}(\sqrt{t/(1+t)})^2}{1- \varphi_{K,2}(\sqrt{t/(1+t)})^2}  \le \lambda(K) \, \max\{  t^{1/K}, t^K\} \,.
\end{equation}
Setting $\sqrt{t/(1+t)} = \th (u/2)$ we have $t=\sh^2 (u/2)$ and the above inequality
\eqref{eta1}
attains the following form
\begin{equation}
\label{eta2}
\eta_K(\sh^2 (u/2)) \equiv \frac{\varphi_{K,2}(\th (u/2))^2}{1- \varphi_{K,2}(\th (u/2))^2}  \le \lambda(K) \, \max\{  (\sh^2 (u/2))^{1/K}, (\sh^2 (u/2))^K\} \,.
\end{equation}

 \begin{nonsec}{\bf Proof of Theorem \ref{th3}.}   \end{nonsec} 
For short we write $\mathbb{H}$ in place of $ \mathbb{H}^2$
 and fix  $x,y \in \mathbb{H}$  and  
 $\rho= \rho_{ \mathbb{H}}(x,y), \, \rho'= \rho_{ \mathbb{H}}(f(x),f(y)).$ 
 Now by   \eqref{hc3} we obtain for $c>0$
 \[ \displaystyle
 h_{ \mathbb{H},c}(f(x),f(y)) = \log(1+ 2 c \,{\rm sh} \frac{\rho'}{2}) 
= \log (1+ 2 c\, \frac{ {{\rm th}\frac{\rho'}{2}}}{{\sqrt{1- {\rm th}^2\frac{\rho'}{2}}}} )
\]
and by the Schwarz lemma for quasiregular mappings \cite[16.2]{hkv}
\[
 h_{ \mathbb{H},c}(f(x),f(y))
\le  \log (1+ 2 c \frac{\varphi_{K,2}({\rm th}
 \frac{\rho}{2})}{\sqrt{1- \varphi_{K,2}({\rm th}\frac{\rho}{2})^2}} ) \,.
\]
By \eqref{eta2} we have
\[
  \log \big(1+ 2 c \frac{\varphi_{K,2}({\rm th}\frac{\rho}{2})}{\sqrt{1- \varphi_{K,2}({\rm th}\frac{\rho}{2})^2}} \big) \le \log \left(1+ 2 c \, \lambda(K)^{1/2}\,
\max \{  ({\rm sh} \frac{\rho}{2})^{1/K},  ({\rm sh} \frac{\rho}{2})^{K} \} \right)
\] 
and further by the Bernoulli inequality \eqref{berineq}
\[
\le \lambda(K)^{1/2}\, \log \left(1+ 2 c \, 
\max \{  ({\rm sh} \frac{\rho}{2})^{1/K},  ({\rm sh} \frac{\rho}{2})^{K} \} \right).
\]  
Finally, by Theorem \ref{fuji}  for $c \geq 1$
\[
\le  K^{1+c}\, \lambda(K)^{1/2}\, \max \{  (\log(1+2c\, {\rm sh} \frac{\rho}{2}))^{1/K}, \log(1+2c \,{\rm sh} \frac{\rho}{2})\}
\]
 \[ =  K^{1+c}\, \lambda(K)^{1/2}\, 
\max \{ h_{ \mathbb{H},c}(x,y)^{1/K}, h_{ \mathbb{H},c}(x,y) \}\,.
\]
\hfill $\square$
\medskip

{\bf Acknowledgements.} We are indebted to the referees for several useful comments.

\vspace{0.5cm}


\begin{thebibliography}{HIMPS}


\bibitem[AVV]{avv} \textsc{G.\,D. Anderson, M.\,K. Vamanamurthy, and M. Vuorinen},
Conformal Invariants, Inequalities, and Quasiconformal Maps, John Wiley \& Sons, New York, 1997.



\bibitem[B]{b} {\sc  A.\,F. Beardon}, The Geometry of Discrete Groups,
 Springer-Verlag, New York, 1983.
 

\bibitem[DHV]{dhv}   
 {\sc O. Dovgoshey, P. Hariri, and  M. Vuorinen},
{Comparison theorems for hyperbolic type metrics}, 
{Complex Var. Elliptic Equ.}
61,  11, (2016), 1464--1480.
 



 \bibitem[FMV]{fmv} {\sc M. Fujimura, M. Mocanu, and M. Vuorinen},
 {Barrlund's distance function and quasiconformal maps}, 
 Complex Var. Elliptic Equ.  66,  8, 1225--1255 (2021).


\bibitem[GH]{gh}  {\sc F.W. Gehring and K. Hag},  The Ubiquitous Quasidisk.
With contributions by Ole Jacob Broch.
American Mathematical Society, Providence, RI, 2012. xii+171 pp.

\bibitem[HKV]{hkv}  {\sc P. Hariri, R. Kl\'en, and M. Vuorinen},
Conformally Invariant Metrics  and Quasiconformal Mappings,
Springer Monographs in Mathematics, Springer, Berlin, 2020.


\bibitem[HIMPS]{ha}{\sc P.A.  H\"ast\"o, Z. Ibragimov, D. Minda,
S. Ponnusamy, S. Sahoo,  }
Isometries of some hyperbolic-type path metrics, and the hyperbolic medial axis.
(English summary) In the tradition of Ahlfors-Bers. IV, 63--74,
Contemp. Math., 432, Amer. Math. Soc., Providence, RI, 2007.


\bibitem[H]{h}{\sc J. Heinonen,} {Lectures on Analysis on Metric Spaces.}
 Springer-Verlag, New York, 2001.

\bibitem[P]{p} {\sc A. Papadopoulos, }  Metric spaces, convexity and non-positive curvature. Second edition. IRMA Lectures in Mathematics and Theoretical Physics,
 6. European Mathematical Society (EMS), Z\"{u}rich, 2014. xii+309 pp.


\bibitem[Ra]{ra} {\sc  O. Rainio, } Inequalities for geometric mean distance metric,
Ukrainian Math. J. 76 (10), 2024, arXiv 2404.01017. 
DOI: https://doi.org/10.3842/umzh.v76i10.7787.
\bibitem[R]{r} {\sc S. Rickman}, Quasiregular mappings,
Ergebnisse der Mathematik und ihrer Grenzgebiete (3) 
[Results in Mathematics and Related Areas (3)], 26. Springer--Verlag, Berlin, 1993.





\bibitem[WWJZ]{wwjz} {\sc Y.
Wu, G. Wang, G. Jia, and X. Zhang}, Lipschitz constants for a hyperbolic type metric under M\"obius transformations. Czech. Math. J. 74, 445--460 (2024). \hfill
\newline
https://doi.org/10.21136/CMJ.2024.0366-23.



\bibitem[ZZPG]{zzpg} {\sc Q.
Zhou, Z. Zheng, S. Ponnusamy, and T. Guan}, Dovgoshey--Hariri--Vuorinen's metric and Gromov hyperbolicity. Arch. Math. 123, 319--327 (2024). https://doi.org/10.1007/s00013--024--02021--w.
\end{thebibliography}
\end{document}